\documentclass[11pt,reqno]{amsart}
\usepackage{enumerate}
\usepackage[T1]{fontenc}
\usepackage[utf8]{inputenc}
\usepackage{lmodern}
\usepackage{amsmath,amssymb,amsthm,mathtools}
\usepackage{xcolor}
\usepackage{tikz}
\usepackage[margin=1.1in]{geometry}
\usepackage{microtype}
\usepackage[colorlinks=true,linkcolor=blue,citecolor=blue,urlcolor=blue]{hyperref}
\hypersetup{
  pdftitle={Note on Robins' Conjecture in All Dimensions},
  pdfauthor={Sinai Robins and Oleg Asipchuk},
  pdfsubject={Fine-lattice reconstruction of convex bodies from sampled Fourier transforms},
  pdfkeywords={Fourier transform, convex body, periodization, lattice sampling, multi-tiling, aliasing}
}

\numberwithin{equation}{section}

\renewcommand{\eqref}[1]{\hyperref[#1]{\textup{(\ref*{#1})}}}

\newtheoremstyle{blueplain}
  {.5\baselineskip plus .2\baselineskip}
  {.5\baselineskip plus .2\baselineskip}
  {\normalcolor\itshape}{}{\bfseries}{.}{.5em}
  {{\color{blue}\thmname{#1}\thmnumber{ #2}}%
   \thmnote{ {\normalfont\color{black}(#3)}}}
\newtheoremstyle{bluedefinition}
  {.5\baselineskip plus .2\baselineskip}
  {.5\baselineskip plus .2\baselineskip}
  {\normalcolor\normalfont}{}{\bfseries}{.}{.5em}
  {{\color{blue}\thmname{#1}\thmnumber{ #2}}%
   \thmnote{ {\normalfont\color{black}(#3)}}}

\theoremstyle{blueplain}
\newtheorem{theorem}{Theorem}
\newtheorem{corollary}{Corollary}
\newtheorem{lemma}{Lemma}

\theoremstyle{plain}

\newtheorem{conjecture}{Conjecture}
\theoremstyle{definition}

\theoremstyle{bluedefinition}
\newtheorem{question}{Question}
\theoremstyle{remark}
\newtheorem{remark}{Remark}

\newcommand{\R}{\mathbb{R}}
\newcommand{\Z}{\mathbb{Z}}
\newcommand{\T}{\mathbb{T}}
\newcommand{\one}{\mathbf{1}}
\newcommand{\diam}{\operatorname{diam}}
\newcommand{\interior}{\operatorname{int}}

\title{Fourier sampling on a lattice}
\author[Oleg Asipchuk and Sinai Robins]{
Oleg Asipchuk 
 and 
Sinai Robins\textsuperscript{*}
}
\thanks{\textsuperscript{*}The author gratefully acknowledges the support given by the
S\~ao Paulo Research Foundation (FAPESP), Grant No.~2023/03167-5.}

\date{August 25, 2026}

\begin{document}

\begin{abstract}
We characterize the ambiguity in reconstructing bounded measurable sets
from their Fourier transforms sampled on a lattice. For a full-rank
lattice \(\mathcal L\subset\R^d\), we prove that two bounded measurable
sets have identical indicator Fourier transforms on \(\mathcal L^*\)
if and only if they are equidecomposable, modulo null sets, into finitely
many measurable pieces by translations in \(\mathcal L\).
If the sets are connected finite unions of convex bodies and no two
distinct points of either set are congruent modulo \(\mathcal L\),
we show that this equidecomposition reduces to a single translation
by a vector of \(\mathcal L\). In particular, two such sets that are
symmetric about the origin must coincide.

We also construct counterexamples in dimensions two and three to a
conjecture of the second author, which had a similar hypothesis, but lacked the anti-aliasing
hypothesis.   Together with the first author's counterexamples in dimensions at least
four, these show that the conjecture fails in every dimension $d\geq 2$.  The positive results here are an outgrowth of refinements of this conjecture.
\end{abstract}

\maketitle


\section{Introduction}

\noindent
Consider the following general Fourier sampling question on a lattice.
\begin{question}
\label{general sampling-hypothesis}
Suppose we are given measurable, bounded sets
\(\mathcal P,Q\subset\R^d\), and a full-rank lattice \(\mathcal L\).
What consequences does the following Fourier sampling assumption have:
\begin{equation}
\label{Question:Fourier equality on a lattice}
\widehat{\one_{\mathcal P}}(n)=\widehat{\one_Q}(n)
\qquad\text{for every }n\in\mathcal L^*?
\end{equation}
\end{question}
Given any body \(\mathcal P\subset\R^d\), its \textbf{difference body}
\(\mathcal P-\mathcal P:=\{a-b\mid a,b\in\mathcal P\}\) is useful in measuring
sparsity. It is immediate that for a lattice \(\mathcal L\) with a
sufficiently large fundamental domain, we have the \textbf{sparse lattice
condition}:
\begin{equation}
\label{eq:main-no-aliasing-hypothesis}
(\mathcal P-\mathcal P)\cap\mathcal L=\{0\},
\end{equation}
because \(\mathcal P\) is bounded. Equivalently, \eqref{eq:main-no-aliasing-hypothesis} means that
no two distinct points of \(\mathcal P\) are congruent modulo \(\mathcal L\).
Our first main result is the following geometric equivalence between the
hypothesis \eqref{Question:Fourier equality on a lattice} and
equidecomposability.
\begin{theorem}
\label{thm:equidecomposition}
Let \(\mathcal L\subset\R^d\) be a full-rank lattice, and let
\(\mathcal P,Q\subset\R^d\) be bounded measurable sets. Then the following
statements are equivalent:
\begin{enumerate}[\rm (a)]
\item $\widehat{\one_{\mathcal P}}(\xi)
=
\widehat{\one_Q}(\xi)$ for every $\xi\in\mathcal L^*$.

\item
The sets \(\mathcal P\) and \(Q\) are measurably equidecomposable, modulo null
sets, by translations in \(\mathcal L\). More precisely, there are finite
measurable partitions
\begin{equation*}
\mathcal P=\mathcal P_1\sqcup\cdots\sqcup\mathcal P_N
\qquad \text{and} \qquad
Q=Q_1\sqcup\cdots\sqcup Q_N
\label{eq:equidecomposition-partitions}
\end{equation*}
modulo null sets, and vectors
\(\ell_1,\ldots,\ell_N\in\mathcal L\), such that
$Q_j=\mathcal P_j+\ell_j$ for $j=1,\ldots,N$.

\end{enumerate}
\end{theorem}


We recall that a measurable set \(\mathcal K\subset\R^d\) is called a {\bf multi-tiler}
  if
there is a discrete translation multiset \(\Lambda\subset\R^d\) such that
$
\sum_{\lambda\in\Lambda}\one_{\mathcal K}(x-\lambda)=k$
for almost every $x\in\R^d$.
 A {\bf body \(\mathcal K\)} is any compact subset of Euclidean space; as usual, a
 convex body is compact, convex, and has nonempty interior.
A conjecture made by the second author is the following.

\begin{conjecture}[Robins]
\label{conjecture Robins}
Let \(\mathcal P,Q\subset\R^d\) be convex, centrally symmetric bodies of positive
measure, neither of which is a multi-tiler.  If
\begin{equation*}
\widehat{\one_{\mathcal P}}(n)=\widehat{\one_Q}(n)
\qquad\text{for every }n\in\Z^d,
\label{eq:original-sampling-hypothesis}
\end{equation*}
then \(\mathcal P\) and \(Q\) agree, up to a rigid motion and a set of measure zero.
\end{conjecture}

The first author disproved Conjecture~\ref{conjecture Robins} in dimensions
\(d\geq4\) in \cite{Asipchuk2026}. Here we show that
Conjecture~\ref{conjecture Robins} is false in the remaining dimensions
\(2\leq d\leq3\). However, there is a nearby strong
positive result for sparse lattices, suggested by Conjecture \ref{conjecture Robins}.

\begin{theorem}
\label{thm:general-lattice-reconstruction}
Let \(\mathcal L\subset\R^d\) be a full-rank lattice.
Let \(\mathcal P,Q\subset\R^d\) be connected finite unions of convex bodies, each
satisfying the sparse lattice property
\eqref{eq:main-no-aliasing-hypothesis} with respect to \(\mathcal L\).
Suppose that
\begin{equation*}
\widehat{\one_{\mathcal P}}(\xi)
=
\widehat{\one_Q}(\xi)
\qquad
\text{for every }\xi\in \mathcal L^*,
\label{eq:general-lattice-sampling}
\end{equation*}
where \(\mathcal L^*\) is the dual lattice of \(\mathcal L\). Then there is
an \(\ell\in\mathcal L\) such that
$Q=\mathcal P+\ell$.
\end{theorem}

\begin{remark}[The origin-symmetric case]
\label{rem:origin-symmetric-consequence}
If, in addition, \(\mathcal P=-\mathcal P\) and \(Q=-Q\), then the conclusion of
Theorem~\ref{thm:general-lattice-reconstruction} strengthens to
$\mathcal P=Q$.
Indeed, if \(Q=\mathcal P+\ell\) with \(\ell\in \mathcal L\), then origin symmetry gives
\(\mathcal P+\ell=\mathcal P-\ell\), and hence \(\mathcal P=\mathcal P-2\ell\).  A nonempty compact set cannot
be invariant under a nonzero translation, so \(\ell=0\).
\hfill \(\lozenge\)
\end{remark}

\begin{remark}
We can give a more intuitive diameter bound by using the length of a shortest vector in the lattice $\mathcal L$,  defined by
$
\lambda_1(\mathcal L)
:=
\min\bigl\{
\|\ell\|:\ell\in \mathcal L\setminus\{0\}
\bigr\}
$.
The diameter bound
\begin{equation*}
\max\{\diam(\mathcal P),\diam(Q)\}
<
\lambda_1(\mathcal L)
\label{eq:general-lattice-diameter-bound}
\end{equation*}
implies the sparse lattice condition
\eqref{eq:main-no-aliasing-hypothesis}.
Indeed, if \(x,y\in \mathcal P\) and \(x-y\in \mathcal L\setminus\{0\}\), then
\begin{equation}
\|x-y\|
\geq
\lambda_1(\mathcal L)
>
\diam(\mathcal P),
\label{eq:shortest-vector-diameter-contradiction}
\end{equation}
which is impossible. The same argument applies to \(Q\).

The same proof of Theorem~\ref{thm:general-lattice-reconstruction} applies
more generally to connected bodies \(\mathcal P,Q\) satisfying
\(\overline{\interior(\mathcal P)}=\mathcal P\) and
\(\overline{\interior(Q)}=Q\), together with the sparse lattice property
\eqref{eq:main-no-aliasing-hypothesis}.
\hfill $\lozenge$
\end{remark}

\begin{remark}
Theorem~\ref{thm:equidecomposition} gives a geometric interpretation
of the ambiguity in lattice sampling: the sampled Fourier data allow
a set to be cut into finitely many measurable pieces, each translated
by a possibly different vector of \(\mathcal L\).
Theorem~\ref{thm:general-lattice-reconstruction} shows how the sparse
lattice property and connectedness eliminate this freedom.
Sparsity ensures that each point of the quotient torus has at most
one representative in either body, while regularity upgrades the
almost-everywhere agreement supplied by equidecomposability to exact
agreement of their torus images. Matching points with the same torus
image therefore gives a continuous correspondence between the bodies.
Its displacement takes values in the discrete lattice \(\mathcal L\),
so connectedness forces that displacement to be constant.
Thus the apparently independent translations of the pieces reduce
to a single translation of the entire body.
\end{remark}


To fix notation, for a measurable set \(\mathcal K\subset\R^d\) of finite measure, we use the
Fourier-transform convention
\begin{equation*}
\widehat{\one_{\mathcal K}}(\xi)
:=
\int_{\mathcal K} e^{-2\pi i\langle x,\xi\rangle}\,dx,
\qquad \xi\in\R^d.
\label{eq:fourier-transform}
\end{equation*}
Fourier transforms of polytopes, Poisson summation, translational tilings,
and multivariate sampling are developed from a common geometric viewpoint
in \cite{Robins2024}; see also \cite{Kolountzakis2004} for the
Fourier-analytic study of translational tilings and
\cite{SteinShakarchi2003} for the underlying Fourier analysis.


\section{Counterexamples in dimensions two and three}
\label{sec:counterexamples-dimensions-two-three}

We now give the planar and three-dimensional counterexamples invoked in the
introduction.  We first recall the standard Fourier criterion for lattice multi-tiling; see,
for example, \cite[Formula~(1.5) in Lecture~1]{Kolountzakis2004}.

\begin{lemma}[Fourier criterion for lattice multi-tiling]
\label{lem:Fourier-criterion-lattice-multitiling}
Let \(\mathcal K\subset\R^d\) be a measurable set of positive finite measure, let
\(\Lambda\subset\R^d\) be a full-rank lattice, and let \(k\in\mathbb{N}\).
Then
\begin{equation*}
\mathcal K\text{ \(k\)-tiles \(\R^d\) by translations with \(\Lambda\)}
\quad\Longleftrightarrow\quad
\widehat{\one_{\mathcal K}}(\xi)=0
\quad\text{for every }\xi\in\Lambda^*\setminus\{0\},
\label{eq:Fourier-criterion-lattice-multitiling}
\end{equation*}
where
\begin{equation*}
\Lambda^*
:=
\bigl\{\xi\in\R^d:\langle\xi,\lambda\rangle\in\Z
\text{ for every }\lambda\in\Lambda\bigr\}
\label{eq:dual-lattice-counterexample-section}
\end{equation*}
is the dual lattice.
\end{lemma}

\begin{figure}[ht]
\centering
\begin{minipage}[t]{0.47\textwidth}
\centering
\begin{tikzpicture}[scale=0.49]

\fill[blue!15]
(-3.5,0.5)--(-2.5,2.5)--(2.5,2.5)--(3.5,0.5)--
({pi+0.5},0)--(3.5,-0.5)--(2.5,-2.5)--(-2.5,-2.5)--
(-3.5,-0.5)--({-pi-0.5},0)--cycle;
\draw[step=1,gray!55,very thin] (-4.5,-3.5) grid (4.5,3.5);
\draw[blue,thick]
(-3.5,0.5)--(-2.5,2.5)--(2.5,2.5)--(3.5,0.5);
\draw[blue,thick]
(3.5,-0.5)--(2.5,-2.5)--(-2.5,-2.5)--(-3.5,-0.5);
\draw[red,thick]
(-3.5,0.5)--({-pi-0.5},0)--(-3.5,-0.5);
\draw[red,thick]
(3.5,0.5)--({pi+0.5},0)--(3.5,-0.5);

\end{tikzpicture}

\smallskip
\textup{(a)} \(\mathcal P=O_1\cup D\)
\end{minipage}
\hfill
\begin{minipage}[t]{0.47\textwidth}
\centering
\begin{tikzpicture}[scale=0.49]

\fill[blue!15]
(-3.5,0.5)--(-0.5,3.5)--(0.5,3.5)--(3.5,0.5)--
({pi+0.5},0)--(3.5,-0.5)--(0.5,-3.5)--(-0.5,-3.5)--
(-3.5,-0.5)--({-pi-0.5},0)--cycle;
\draw[step=1,gray!55,very thin] (-4.5,-4.5) grid (4.5,4.5);
\draw[blue,thick]
(-3.5,0.5)--(-0.5,3.5)--(0.5,3.5)--(3.5,0.5);

\draw[blue,thick]
(3.5,-0.5)--(0.5,-3.5)--(-0.5,-3.5)--(-3.5,-0.5);
\draw[red,thick]
(-3.5,0.5)--({-pi-0.5},0)--(-3.5,-0.5);
\draw[red,thick]
(3.5,0.5)--({pi+0.5},0)--(3.5,-0.5);
\end{tikzpicture}

\smallskip
\textup{(b)} \(Q=O_2\cup D\)
\end{minipage}
\caption{The bodies \(\mathcal P=O_1\cup D\) and \(Q=O_2\cup D\).  The octagonal
parts of the boundaries are shown in blue, the two additional triangular
parts are shown in red, and the grid has unit spacing.}
\label{fig:planar-counterexample-bodies}
\end{figure}
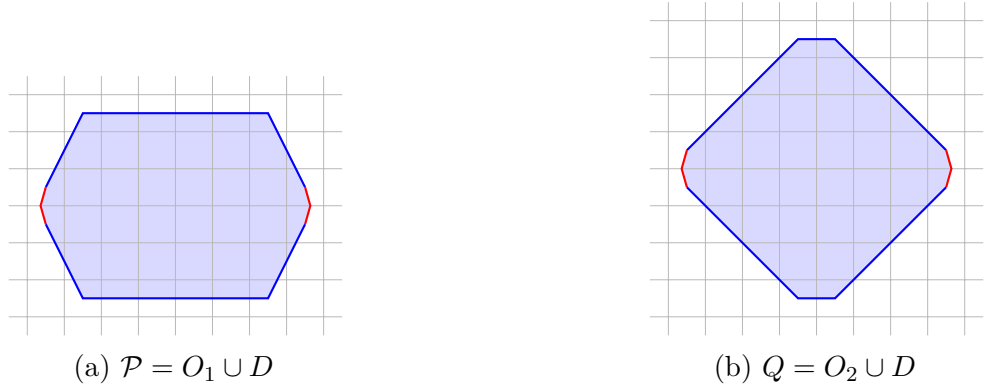

We shall also use Bolle's characterization of multiple lattice tiles
\cite{Bolle1994}.

\begin{theorem}[Bolle]
\label{thm:Bolle-multiple-lattice-tile}
Let \(\mathcal K\subset\R^2\) be a centrally symmetric convex polygon, and let
\(\Lambda\subset\R^2\) be a lattice.  Then \(\mathcal K\) is a multiple lattice tile
with respect to \(\Lambda\), with some multiplicity \(w\in\mathbb{N}\), if
and only if the following conditions hold for every edge \(e\) of \(\mathcal K\):
\begin{enumerate}
\item[\textup{(i)}]
The relative interior of \(e\) contains a point of
\(\tfrac12\Lambda\).

\item[\textup{(ii)}]
If the midpoint of \(e\) does not belong to \(\tfrac12\Lambda\), then the
edge vector of \(e\) belongs to \(\Lambda\).
\end{enumerate}
\end{theorem}

\begin{theorem}
\label{prop:counterexamples-dimensions-two-three}
Conjecture \ref{conjecture Robins}
 is false in dimensions \(d=2\) and \(d=3\).
\end{theorem}

\begin{proof}
We begin with \(d=2\).  Let \(O_1\) be the octagon
\begin{equation*}
\begin{aligned}
O_1:=\operatorname{conv}\bigl\{&
(-\tfrac72,\tfrac12),
(-\tfrac52,\tfrac52),
(\tfrac52,\tfrac52),
(\tfrac72,\tfrac12),
(\tfrac72,-\tfrac12),
(\tfrac52,-\tfrac52),
(-\tfrac52,-\tfrac52),
(-\tfrac72,-\tfrac12)
\bigr\},
\end{aligned}
\label{eq:first-counterexample-octagon}
\end{equation*}
and let \(O_2\) be the octagon
\begin{equation*}
\begin{aligned}
O_2:=\operatorname{conv}\bigl\{&
(-\tfrac72,\tfrac12),
(-\tfrac12,\tfrac72),
(\tfrac12,\tfrac72),
(\tfrac72,\tfrac12),
(\tfrac72,-\tfrac12),
(\tfrac12,-\tfrac72),
(-\tfrac12,-\tfrac72),
(-\tfrac72,-\tfrac12)
\bigr\}.
\end{aligned}
\label{eq:second-counterexample-octagon}
\end{equation*}
We define the two small triangles
\begin{equation*}
\begin{aligned}
D:=
&\operatorname{conv}\bigl\{
(-\tfrac72,\tfrac12),
(-\tfrac72,-\tfrac12),
(-\pi-\tfrac12,0)
\bigr\}
\cup
\operatorname{conv}\bigl\{
(\tfrac72,\tfrac12),
(\tfrac72,-\tfrac12),
(\pi+\tfrac12,0)
\bigr\},
\end{aligned}
\label{eq:counterexample-domain-D}
\end{equation*}
and set
\begin{equation*}
\mathcal P:=O_1\cup D
\qquad \text{and} \qquad 
Q:=O_2\cup D.
\label{eq:planar-counterexample-bodies}
\end{equation*}
These are the two convex decagons shown in
Figure~\ref{fig:planar-counterexample-bodies}; moreover, \(\mathcal P=-\mathcal P\) and
\(Q=-Q\).

There are two types of edges in these decagons.  The blue edges have
endpoints with rational coordinates, whereas each red edge has one endpoint
whose first coordinate is irrational.  Thus the midpoint of every blue edge
has rational coordinates, while the midpoint and the edge vector of every
red edge have irrational first coordinate.  In particular, the rational
midpoints of two linearly independent blue edges force any lattice
containing their doubles to be rational.  The red edges therefore prevent
the alternatives in Theorem~\ref{thm:Bolle-multiple-lattice-tile} from
holding simultaneously.  Consequently, neither \(\mathcal P\) nor \(Q\) is a
multiple lattice tile.

The octagons \(O_1\) and \(O_2\) each \(31\)-tile \(\R^2\) with respect to
the lattice \(\Z^2\), and both have area \(31\).  Hence
Lemma~\ref{lem:Fourier-criterion-lattice-multitiling} gives
\begin{equation*}
\widehat{\one_{O_1}}(\xi)
=
\widehat{\one_{O_2}}(\xi)
=
\begin{cases}
31,&\xi=(0,0),
\\
0,&\xi\in\Z^2\setminus\{(0,0)\}.
\end{cases}
\label{eq:octagon-Fourier-values}
\end{equation*}
Since the overlaps between \(O_i\) and \(D\) have measure zero, it follows
that, for every \(\xi\in\Z^2\),
\begin{equation}
\widehat{\one_{\mathcal P}}(\xi)
=
\widehat{\one_{O_1}}(\xi)+\widehat{\one_D}(\xi)
=
\widehat{\one_{O_2}}(\xi)+\widehat{\one_D}(\xi)
=
\widehat{\one_Q}(\xi).
\label{eq:planar-counterexample-Fourier-equality}
\end{equation}
The bodies are not congruent: \(\mathcal P\) has edges of length \(5\), whereas
\(Q\) has no edge of that length.  This proves the failure of the conjecture
in dimension two.

For dimension three, take the origin-centered prisms
\begin{equation*}
\mathcal P_3:=\mathcal P\times[-\tfrac12,\tfrac12]
\qquad \text{and} \qquad 
Q_3:=Q\times[-\tfrac12,\tfrac12].
\label{eq:three-dimensional-counterexample-prisms}
\end{equation*}
The product formula for Fourier transforms and
\eqref{eq:planar-counterexample-Fourier-equality} give $\widehat{\one_{\mathcal P_3}}(\xi)
=
\widehat{\one_{Q_3}}(\xi)$
for every $\xi\in\Z^3$.

The prisms are centrally symmetric, noncongruent, and inherit the failure to
be multiple lattice tiles from their planar bases.  Thus the conjecture also
fails in dimension three.
\end{proof}

\section{Periodization on a lattice torus}
\label{sec:lattice-periodization}

Let \(\mathcal L\subset\R^d\) be a full-rank lattice.  We let
$
\T:=\R^d/\mathcal L
$
be the associated flat torus, and we let $\pi_{\mathcal L}:\R^d\longrightarrow\T$ be the standard quotient map.  Let \(F\) be a measurable fundamental domain for
\(\mathcal L\), and write \(\operatorname{covol}(\mathcal L):=|F|\).  If
\(f\in L^1(\R^d)\) has compact support, its \(\mathcal L\)-periodization is
\begin{equation}
\mathcal{P}_{\mathcal L} f(x)
:=
\sum_{\ell\in \mathcal L}f(x+\ell).
\label{eq:general-periodization-definition}
\end{equation}
The sum in \eqref{eq:general-periodization-definition} is finite at every
point and defines an element of \(L^1(\T)\).

\begin{lemma}[Periodization identity]
\label{lem:general-periodization-identity}
Let \(F\subset\R^d\) be any measurable, bounded fundamental domain for \(\mathcal L\).
For every \(\xi\in \mathcal L^*\), the Fourier coefficient of
\(\mathcal{P}_{\mathcal L} f\) corresponding to the character
\(x\mapsto e^{2\pi i\langle x,\xi\rangle}\) is
\begin{equation}
\frac{1}{\operatorname{covol}(\mathcal L)}
\int_F \mathcal{P}_{\mathcal L} f(x)e^{-2\pi i\langle x,\xi\rangle}\,dx
=
\frac{1}{\operatorname{covol}(\mathcal L)}\widehat f(\xi).
\label{eq:general-periodization-fourier-coefficient}
\end{equation}
Consequently, if two compactly supported integrable functions have equal
Fourier transforms on \(\mathcal L^*\), then their \(\mathcal L\)-periodizations agree almost
everywhere on \(\T\).
\end{lemma}

\begin{proof}
Since \(\langle\ell,\xi\rangle\in\Z\) for every \(\ell\in \mathcal L\) and
\(\xi\in \mathcal L^*\), unfolding the integral over the translates of \(F\) gives
\begin{align*}
\int_F \mathcal{P}_{\mathcal L} f(x)e^{-2\pi i\langle x,\xi\rangle}\,dx
&=
\sum_{\ell\in \mathcal L}
\int_F f(x+\ell)e^{-2\pi i\langle x,\xi\rangle}\,dx
\\
&=
\sum_{\ell\in \mathcal L}
\int_{F+\ell} f(y)e^{-2\pi i\langle y,\xi\rangle}\,dy
\\
&=
\int_{\R^d}f(y)e^{-2\pi i\langle y,\xi\rangle}\,dy
\\
&=
\widehat f(\xi).
\end{align*}
Compact support makes the sum locally finite.  Dividing by
\(\operatorname{covol}(\mathcal L)\) proves
\eqref{eq:general-periodization-fourier-coefficient}.  The final assertion
follows from uniqueness of Fourier series in \(L^1(\T)\); see, for
example, \cite{Katznelson2004,SteinShakarchi2003}.
\end{proof}


\bigskip

\section{Proof of Theorem~\ref{thm:equidecomposition}}
\label{sec:equidecomposition}

Intuitively, sampling on \(\mathcal L^*\) records exactly the number
of points of the set in each fiber of \(\R^d\to\R^d/\mathcal L\).
It forgets only which lattice translate contains each point, and finite
measurable equidecomposition by vectors of \(\mathcal L\) is precisely
the freedom left by that loss of information.  The proof below makes these ideas more precise.

\begin{proof}[Proof of Theorem \ref{thm:equidecomposition}]
Choose an invertible linear map \(A:\R^d\to\R^d\) such that
\(\mathcal L=A\Z^d\); then \(\mathcal L^*=A^{-T}\Z^d\). For a bounded
measurable set \(E\subset\R^d\), put \(F:=A^{-1}E\). A change of
variables gives
\begin{equation*}
\widehat{\one_{F}}(k)
=
\frac{1}{|\det A|}
\widehat{\one_E}(A^{-T}k),
\qquad k\in\Z^d.
\label{eq:lattice-coordinate-fourier-transform}
\end{equation*}
Thus condition~\textup{(a)} says exactly that
\(\widehat{\one_{\mathcal P'}}\) and \(\widehat{\one_{Q'}}\) agree on
\(\Z^d\), where $\mathcal P' :=A^{-1}\mathcal P$ and $Q' :=A^{-1} Q$. Lemma~\ref{lem:general-periodization-identity}, applied to
the lattice \(\Z^d\), therefore yields equality of their periodizations
on \(\R^d/\Z^d\). Moreover,
\begin{equation*}
\mathcal{P}_{\Z^d}\one_{F}(u)
=
\sum_{n\in\Z^d}\one_E(Au+An)
=
\mathcal{P}_{\mathcal L}\one_E(Au).
\label{eq:equidecomposition-coordinate-periodization}
\end{equation*}
Now apply the preceding identity first with \(E=\mathcal P\) and
\(F=\mathcal P'=A^{-1}\mathcal P\), and then with \(E=Q\) and
\(F=Q'=A^{-1}Q\). The almost-everywhere equality of the
\(\Z^d\)-periodizations therefore gives
\begin{equation*}
\mathcal{P}_{\mathcal L}\one_{\mathcal P}(Au)
=
\mathcal{P}_{\mathcal L}\one_Q(Au)
\qquad\text{for almost every }u\in\R^d.
\end{equation*}
Setting \(x=Au\), and using that the invertible linear map \(A\)
preserves null sets, we obtain
\begin{equation}
\mathcal{P}_{\mathcal L}\one_{\mathcal P}
=
\mathcal{P}_{\mathcal L}\one_Q
    \label{eq:equidecomposition-equal-periodizations}
\end{equation}
almost everywhere on $\R^d/\mathcal L$, where $\mathcal{P}_{\mathcal L}$ is defined as in \eqref{eq:general-periodization-definition}.

Geometrically, the quotient map
\(\pi_{\mathcal L}:\R^d\to\R^d/\mathcal L\) folds Euclidean space along
the lattice. By definition, the value of
\(\mathcal{P}_{\mathcal L}\one_E\) at a torus point counts the points
of \(E\) in the corresponding lattice fiber.
Under no aliasing that count is at most one. Here several sheets of \(E\)
may lie over the same torus point, and
\eqref{eq:equidecomposition-equal-periodizations} says that \(\mathcal P\)
and \(Q\) have the same number of sheets over almost every point.

We now match those sheets. Let \(D=A[0,1)^d\) be a half-open fundamental
parallelepiped for \(\mathcal L\). Since \(\mathcal P\) and \(Q\) are
bounded, the sets
\begin{equation*}
\Lambda_{\mathcal P}
:=\{\lambda\in\mathcal L:(D+\lambda)\cap\mathcal P\neq\varnothing\}
\qquad \text{and} \qquad\Lambda_Q
:=\{\mu\in\mathcal L:(D+\mu)\cap Q\neq\varnothing\}
\label{eq:equidecomposition-finite-translate-sets}
\end{equation*}
are finite. For \(x\in D\), define the finite fibers
\begin{equation}
A_x:=\{\lambda\in\Lambda_{\mathcal P}:x+\lambda\in\mathcal P\}
\qquad \text{and} \qquad
B_x:=\{\mu\in\Lambda_Q:x+\mu\in Q\}.
\label{eq:equidecomposition-fibers}
\end{equation}
Equation \eqref{eq:equidecomposition-equal-periodizations} is precisely
\begin{equation}
|A_x|=|B_x|
\qquad\text{for almost every }x\in D.
\label{eq:equidecomposition-equal-fiber-cardinalities}
\end{equation}

Fix an ordering of each of the finite sets \(\Lambda_{\mathcal P}\)
and \(\Lambda_Q\).
Whenever \eqref{eq:equidecomposition-equal-fiber-cardinalities}
holds, pair the first element of \(A_x\) with the first element of
\(B_x\), the second with the second, and so on. For
\(\lambda\in\Lambda_{\mathcal P}\) and \(\mu\in\Lambda_Q\), let
\(E_{\lambda,\mu}\subset D\) be the set of points \(x\) for which
\(\lambda\) and \(\mu\) are paired. This set is measurable: whether a
given lattice vector has a specified rank in its fiber is determined by
finitely many measurable membership conditions of the form
\(x+\lambda'\in\mathcal P\) or \(x+\mu'\in Q\).

Set $\mathcal P_{\lambda,\mu}:=E_{\lambda,\mu}+\lambda$ and $Q_{\lambda,\mu}:=E_{\lambda,\mu}+\mu$. These sets partition \(\mathcal P\) and \(Q\), respectively, modulo the
null set lying over the exceptional set in
\eqref{eq:equidecomposition-equal-fiber-cardinalities}. Moreover, $Q_{\lambda,\mu}
=
\mathcal P_{\lambda,\mu}+(\mu-\lambda)$, where $\mu-\lambda\in\mathcal L$. Thus each matched pair of sheets differs by a translation within a single
lattice fiber. Since there are only finitely many pairs
\((\lambda,\mu)\), deleting empty pieces and relabeling proves
\textup{(b)}.

Conversely, suppose \textup{(b)} holds. If
\(Q_j=\mathcal P_j+\ell_j\) with \(\ell_j\in\mathcal L\), then for
every \(\xi\in\mathcal L^*\),
\begin{equation*}
\widehat{\one_{Q_j}}(\xi)
=
e^{-2\pi i\langle\ell_j,\xi\rangle}
\widehat{\one_{\mathcal P_j}}(\xi)
=
\widehat{\one_{\mathcal P_j}}(\xi),
\label{eq:equidecomposition-piece-fourier-transform}
\end{equation*}
because \(\langle\ell_j,\xi\rangle\in\Z\). Summing over the finite
partitions gives
\(\widehat{\one_Q}(\xi)=\widehat{\one_{\mathcal P}}(\xi)\) for every
\(\xi\in\mathcal L^*\), proving \textup{(a)}.
\end{proof}

\section{Proof of Theorem~\ref{thm:general-lattice-reconstruction}}
\label{sec:proof-general-lattice-reconstruction}

\begin{lemma}[No aliasing]
\label{lem:general-no-aliasing}
Let \(\mathcal K\subset\R^d\) be compact.  The restriction of \(\pi_{\mathcal L}\) to \(\mathcal K\)
is injective if and only if
\begin{equation*}
(\mathcal K-\mathcal K)\cap \mathcal L=\{0\}.
\label{eq:general-no-aliasing-condition}
\end{equation*}
Under this condition,
\begin{equation}
\mathcal{P}_{\mathcal L}\one_{\mathcal K}
=
\one_{\pi_{\mathcal L}(\mathcal K)}
\qquad\text{on }\T.
\label{eq:general-periodization-indicator-projection}
\end{equation}
\end{lemma}

\begin{proof}
For \(x,y\in \mathcal K\), one has
\begin{equation*}
\pi_{\mathcal L}(x)=\pi_{\mathcal L}(y)
\quad\Longleftrightarrow\quad
x-y\in \mathcal L.
\label{eq:general-same-torus-class}
\end{equation*}
This proves the first assertion.  Under injectivity, each class in
\(\T\) has at most one representative in \(\mathcal K\).  The left-hand side of
\eqref{eq:general-periodization-indicator-projection} counts the
representatives of the class \(\pi_{\mathcal L}(x)\) that lie in \(\mathcal K\); it is therefore
one on \(\pi_{\mathcal L}(\mathcal K)\) and zero off \(\pi_{\mathcal L}(\mathcal K)\).
\end{proof}

\begin{lemma}[Regularity of the projected sets]
\label{lem:regularity-of-projected-sets}
Let \(\mathcal K\subset\R^d\) be a finite union of convex bodies, and suppose that
\(\pi_{\mathcal L}|_{\mathcal K}\) is injective.  Then
\begin{equation}
\pi_{\mathcal L}(\mathcal K)
=
\overline{\interior\bigl(\pi_{\mathcal L}(\mathcal K)\bigr)}.
\label{eq:regular-closed-projection}
\end{equation}
\end{lemma}

\begin{proof}
Because every convex body is the closure of its interior, a finite union
\(\mathcal K\) of convex bodies satisfies
\begin{equation}
\mathcal K=\overline{\interior(\mathcal K)}.
\label{eq:regular-closed-finite-union}
\end{equation}
The quotient map \(\pi_{\mathcal L}\) is a local homeomorphism.  Therefore, if
\(x\in\interior(\mathcal K)\), a sufficiently small open neighborhood of \(x\) is
mapped to an open subset of \(\pi_{\mathcal L}(\mathcal K)\), and hence
\begin{equation}
\pi_{\mathcal L}\bigl(\interior(\mathcal K)\bigr)
\subseteq
\interior\bigl(\pi_{\mathcal L}(\mathcal K)\bigr).
\label{eq:interior-projection-inclusion}
\end{equation}
To prove \eqref{eq:regular-closed-projection}, let
\(y\in\pi_{\mathcal L}(\mathcal K)\), and choose \(x\in\mathcal K\)
such that \(\pi_{\mathcal L}(x)=y\). Let \(U\) be any open
neighborhood of \(y\) in \(\T\). By continuity,
\(\pi_{\mathcal L}^{-1}(U)\) is an open neighborhood of \(x\).
Equation~\eqref{eq:regular-closed-finite-union} implies that this
neighborhood meets \(\interior(\mathcal K)\). Consequently,
\(U\) meets \(\pi_{\mathcal L}\bigl(\interior(\mathcal K)\bigr)\),
and therefore also meets
\(\interior\bigl(\pi_{\mathcal L}(\mathcal K)\bigr)\) by
\eqref{eq:interior-projection-inclusion}. Since every open
neighborhood of \(y\) meets this interior, \(y\) belongs to its
closure. Thus
\begin{equation*}
\pi_{\mathcal L}(\mathcal K)
\subseteq
\overline{\interior\bigl(\pi_{\mathcal L}(\mathcal K)\bigr)}.
\end{equation*}
Conversely, \(\mathcal K\) is compact, so its continuous image
\(\pi_{\mathcal L}(\mathcal K)\) is compact and therefore closed in the
 torus \(\T\). Since
\(\interior\bigl(\pi_{\mathcal L}(\mathcal K)\bigr)
\subseteq\pi_{\mathcal L}(\mathcal K)\), taking closures gives
\begin{equation*}
\overline{\interior\bigl(\pi_{\mathcal L}(\mathcal K)\bigr)}
\subseteq
\pi_{\mathcal L}(\mathcal K).
\end{equation*}
The two inclusions establish \eqref{eq:regular-closed-projection}.
\end{proof}

\begin{proof}[Proof of Theorem \ref{thm:general-lattice-reconstruction}]
By Theorem~\ref{thm:equidecomposition}, the sets
\(\mathcal P\) and \(Q\) are measurably equidecomposable,
modulo null sets, by translations in \(\mathcal L\).
Since these translations leave the quotient map
\(\pi_{\mathcal L}\) unchanged, their torus images agree
up to a set of Haar measure zero. Here we use the fact
that \(\pi_{\mathcal L}\) maps Lebesgue-null sets to
Haar-null sets.

Set $A:=\pi_{\mathcal L}(\mathcal P)$ and $B:=\pi_{\mathcal L}(Q)$. Both sets are compact, and
Lemma~\ref{lem:regularity-of-projected-sets} gives $A=\overline{\interior(A)}$
and $B=\overline{\interior(B)}$.

If \(\interior(A)\setminus B\) were nonempty, it would
be a nonempty open subset of the torus, hence would
have positive Haar measure, contradicting the
almost-everywhere equality of \(A\) and \(B\).
Thus \(\interior(A)\subseteq B\), and closedness of
\(B\) gives \(A\subseteq B\). Interchanging the sets
yields $\pi_{\mathcal L}(\mathcal P)=\pi_{\mathcal L}(Q)$.

The sparse lattice property makes the restrictions
of \(\pi_{\mathcal L}\) to \(\mathcal P\) and \(Q\)
injective. These restrictions are therefore
homeomorphisms onto their common image, since their
domains are compact and the torus is Hausdorff.
Consequently,
\begin{equation*}
\Phi:=
(\pi_{\mathcal L}|_Q)^{-1}
\circ(\pi_{\mathcal L}|_{\mathcal P})
:\mathcal P\longrightarrow Q
\label{eq:general-lift-homeomorphism}
\end{equation*}
is a homeomorphism, and $\Phi(x)-x\in \mathcal L$
for every $x\in \mathcal P$.

The map \(x\mapsto\Phi(x)-x\) is continuous from the
connected set \(\mathcal P\) into the discrete set
\(\mathcal L\), so it is constant. Hence, for some
\(\ell\in\mathcal L\), $Q=\Phi(\mathcal P)=\mathcal P+\ell$,
as required.
\end{proof}


\section{Lattice multi-tiling and polytopal equidecomposability}
\label{sec:lev-liu-consequences}

The corollaries in this section are elementary consequences and specializations of
Theorem~\ref{thm:equidecomposition} and its proof, to piece-wise linear objects.
A \emph{generalized polytope} is a finite union of full-dimensional
convex polytopes; convexity and connectedness of the union are not required.

Theorem~\ref{thm:equidecomposition} also gives the connection between
lattice multi-tiling and equidecomposability established for polytopes
by Lev and Liu \cite[Proposition~1.3]{LevLiu2019}. We first explain why,
when the original sets are generalized polytopes, the measurable pieces in our
construction can be chosen to be polytopes.
All partitions and equidecompositions are understood modulo sets of
Lebesgue measure zero.

\begin{corollary}[Polytopal pieces]
\label{cor:polytopal-pieces}
Let \(\mathcal P,Q\subset\R^d\) be generalized polytopes, and let
\(\mathcal L\) be a full-rank lattice. Their indicator Fourier
transforms agree on \(\mathcal L^*\) if and only if they admit finite
partitions into convex polytopes whose corresponding pieces differ
by translations in \(\mathcal L\).
\end{corollary}

\begin{proof}
The reverse implication is part of
Theorem~\ref{thm:equidecomposition}. For the forward implication, use
the fundamental parallelepiped \(D\) and the finite sets
\(\Lambda_{\mathcal P},\Lambda_Q\) from its proof. Write each of
\(\mathcal P\) and \(Q\) as a finite union of convex polytopes.
Subdivide \(\overline D\) by the supporting hyperplanes of all facets
of these constituent polytopes translated by \(-\lambda\), for
\(\lambda\in\Lambda_{\mathcal P}\), or by \(-\mu\), for
\(\mu\in\Lambda_Q\), respectively.

The resulting full-dimensional open cells are finite in number, and
their closures are convex polytopes. On each such cell \(C\), every
membership condition \(x+\lambda\in\mathcal P\) and
\(x+\mu\in Q\) is constant. Thus the fibers \(A_x\) and \(B_x\)
defined in \eqref{eq:equidecomposition-fibers} are constant on \(C\).
Their cardinalities agree almost everywhere by
\eqref{eq:equidecomposition-equal-fiber-cardinalities}; since \(C\)
has positive measure, these two constant cardinalities must agree.
Fix the same orderings as in the proof of
Theorem~\ref{thm:equidecomposition}, and match the fibers on \(C\).
For every matched pair \((\lambda,\mu)\), take the pieces $\overline C+\lambda$ and
\begin{equation*}
\overline C+\mu
=(\overline C+\lambda)+(\mu-\lambda).
\label{eq:polytopal-cell-matching}
\end{equation*}
These pieces lie in \(\mathcal P\) and \(Q\), respectively, by
closedness. They partition the two sets modulo null sets: the only
points omitted or counted more than once lie on the finitely many
translated cell boundaries. This gives the required polytopal
equidecomposition.
\end{proof}

\begin{corollary}[Multi-tiling and equidecomposability, Lev--Liu {\cite[Proposition~1.3]{LevLiu2019}}]
\label{cor:multitiling-equidecomposition}
Let \(\mathcal P,Q\subset\R^d\) be bounded measurable sets, and
let \(\mathcal L\) be a full-rank lattice. Suppose that \(Q\)
\(k\)-tiles by translations in \(\mathcal L\), where
\(k\in\mathbb N\) is positive. Then \(\mathcal P\) also
\(k\)-tiles by \(\mathcal L\) if and only if \(\mathcal P\)
and \(Q\) are finitely measurably equidecomposable by translations
in \(\mathcal L\). If both sets are generalized polytopes, the pieces can be
chosen to be convex polytopes.
\end{corollary}

\begin{proof}
If both sets \(k\)-tile by \(\mathcal L\), their periodizations
satisfy $\mathcal{P}_{\mathcal L}\one_{\mathcal P}
=k
=\mathcal{P}_{\mathcal L}\one_Q$ almost everywhere on $\T$. Lemma~\ref{lem:general-periodization-identity} gives equality of
their Fourier transforms on \(\mathcal L^*\), including at zero;
in particular,
\begin{equation*}
\operatorname{vol}(\mathcal P)
=\operatorname{vol}(Q)
=k\operatorname{covol}(\mathcal L).
\label{eq:multitiles-equal-volumes}
\end{equation*}
Theorem~\ref{thm:equidecomposition} now supplies the measurable
equidecomposition. Conversely, such an equidecomposition gives
equality of the sampled Fourier transforms by that theorem, and
then equality of periodizations by
Lemma~\ref{lem:general-periodization-identity}. Since the
periodization of \(\one_Q\) is \(k\) almost everywhere, the same
holds for \(\one_{\mathcal P}\). The assertion about polytopal
pieces follows from Corollary~\ref{cor:polytopal-pieces}.
\end{proof}

For generalized polytopes, Corollary~\ref{cor:multitiling-equidecomposition}
recovers exactly \cite[Proposition~1.3]{LevLiu2019}; its measurable
form applies to arbitrary bounded measurable sets. In particular,
two bounded measurable multi-tiles of the same lattice and the same
multiplicity are equidecomposable by lattice translations.

\begin{corollary}[Rearrangement into fundamental parallelepipeds]
\label{cor:fundamental-parallelepipeds}
Let \(\mathcal P\subset\R^d\) be bounded and measurable, let
\(\mathcal L\) be a full-rank lattice, and let \(k\geq1\) be an
integer. Fix a half-open fundamental parallelepiped \(D\) of
\(\mathcal L\) and distinct vectors
\(\ell_1,\ldots,\ell_k\in\mathcal L\). Then \(\mathcal P\)
\(k\)-tiles by \(\mathcal L\) if and only if it is finitely
measurably equidecomposable, by translations in \(\mathcal L\),
with
\begin{equation*}
Q:=\bigsqcup_{j=1}^{k}(D+\ell_j).
\label{eq:fundamental-parallelepipeds-target}
\end{equation*}
If \(\mathcal P\) is a generalized polytope, the equidecomposition can be
chosen to have convex polytopal pieces, with boundaries disregarded.
\end{corollary}

\begin{proof}
Each translate \(D+\ell_j\) is a fundamental domain for
\(\mathcal L\), so \(Q\) \(k\)-tiles by \(\mathcal L\).
Apply Corollary~\ref{cor:multitiling-equidecomposition}. For the
polytopal assertion, replace each half-open parallelepiped by its
closure, which changes only a null set, and apply
Corollary~\ref{cor:polytopal-pieces}.
\end{proof}

The polytopal case of this last corollary is the specialization
discussed immediately after \cite[Proposition~1.3]{LevLiu2019}.
These consequences concern the relation between multi-tiling and
equidecomposability. Lev and Liu's deeper characterization by
lattice Hadwiger functionals \cite[Theorem~1.2]{LevLiu2019} is a
separate result and is not proved here.

\section{Remarks}
\label{sec:remarks}

\begin{remark}
If we let \(\mathcal L=m\Z^d\), then \(\mathcal L^*=m^{-1}\Z^d\) and
\(\lambda_1(\mathcal L)=m\). Then the diameter
bound
\begin{equation*}
m>\max\{\diam(\mathcal P),\diam(Q)\}
\label{eq:coordinate-lattice-diameter-specialization}
\end{equation*}
implies the sparse lattice property for both bodies. But importantly, the general-lattice
formulation shows that the mechanism is intrinsic: the sampling lattice is
the dual of the spatial lattice relative to which aliasing is excluded.
\hfill \(\lozenge\)
\end{remark}

\begin{remark}[Boundary-wrapping ambiguity]
The sparse lattice property in
Theorem~\ref{thm:general-lattice-reconstruction} cannot be weakened to
\begin{equation}
\interior(\mathcal K-\mathcal K)\cap \mathcal L=\{0\}.
\label{eq:interior-sparse-lattice-condition}
\end{equation}
Indeed, take \(\mathcal L=\Z\), \(\mathcal P=[0,1]\), and \(Q=[a,a+1]\), where
\(a\notin\Z\).  The two intervals have identical Fourier transforms on
\(\mathcal L^*=\Z\), but \(Q\neq \mathcal P+\ell\) for every \(\ell\in \mathcal L\).  Here
\eqref{eq:interior-sparse-lattice-condition} holds for both intervals,
whereas the full difference sets contain the nonzero lattice points
\(\pm1\).  This is precisely the boundary-wrapping ambiguity excluded by
\eqref{eq:main-no-aliasing-hypothesis}.
\hfill \(\lozenge\)
\end{remark}

\begin{remark}[The unavoidable translation ambiguity]
The conclusion of Theorem~\ref{thm:general-lattice-reconstruction} is best
possible without an additional normalization.  If \(Q=\mathcal P+\ell_0\) for some
\(\ell_0\in \mathcal L\), then for every \(\xi\in \mathcal L^*\),
\begin{align*}
\widehat{\one_Q}(\xi)
&=
e^{-2\pi i\langle\ell_0,\xi\rangle}
\widehat{\one_{\mathcal P}}(\xi)
=
\widehat{\one_{\mathcal P}}(\xi),
\label{eq:general-unavoidable-translation-ambiguity-two}
\end{align*}
because \(\langle\ell_0,\xi\rangle\in\Z\).
\hfill \(\lozenge\)
\end{remark}

\bigskip
\bigskip


\end{document}